\documentclass[11pt]{article}
\usepackage{amsthm}
\usepackage{amsfonts}
\usepackage{amsmath}
\usepackage{colonequals}
\usepackage{epsfig}
\usepackage{url}
\newcommand{\GG}{{\cal G}}
\newcommand{\CC}{{\cal C}}
\newcommand{\PP}{{\cal P}}

\newtheorem{conjecture}{Conjecture}
\newtheorem{theorem}{Theorem}
\newtheorem{corollary}[theorem]{Corollary}
\newtheorem{lemma}[theorem]{Lemma}

\newcommand{\tw}{\mbox{tw}}

\newcommand{\bal}[2]{$#1\bullet^{#2}$}

\title{Strongly sublinear separators and polynomial expansion}
\author{Zden\v{e}k Dvo\v{r}\'ak\thanks{Computer Science Institute, Charles University, Prague, Czech Republic. E-mail: {\tt rakdver@iuuk.mff.cuni.cz}.
Supported in part by (FP7/2007-2013)/ERC Consolidator grant LBCAD no. 616787.}\and Sergey Norin\thanks{Department of Mathematics and Statistics, McGill
University. Email: {\tt snorin@math.mcgill.ca}. Supported by an NSERC Discovery grant.}}
\date{}

\begin{document}
\maketitle

\begin{abstract}
A result of Plotkin, Rao, and Smith implies that graphs with polynomial expansion
have strongly sublinear separators. We prove a converse of this result
showing that hereditary classes of graphs with strongly sublinear
separators have polynomial expansion. This confirms a conjecture of the
first author.
\end{abstract}

\section{Introduction}
The concept of graph classes with bounded expansion was introduced by Ne\v{s}et\v{r}il and Ossona de Mendez~\cite{grad1}
as a way of formalizing the notion of sparse graph classes.  Let us give a few definitions.

For a graph $G$, a \emph{$k$-minor of $G$} is any graph obtained from $G$ by contracting
pairwise vertex-disjoint subgraphs of radius at most $k$ and removing vertices and edges.  Thus, a $0$-minor is just a subgraph of $G$.
Let us define $\nabla_k(G)$ as $$\max\left\{\frac{|E(G')|}{|V(G')|}:\mbox{$G'$ is a $k$-minor of $G$}\right\}.$$
For a function $f:\mathbf{Z}_0^+\to \mathbf{R}_0^+$, we say that an expansion of a graph $G$ is \emph{bounded by $f$} if
$\nabla_k(G)\le f(k)$ for every $k\ge 0$.  We say that a class $\GG$ of graphs has \emph{bounded expansion} if there exists
a function $f:\mathbf{Z}_0^+\to \mathbf{R}_0^+$ such that $f$ bounds the expansion of every graph in $\GG$.
If such a function $f$ is a polynomial, we say that $\GG$ has \emph{polynomial expansion}.

The definition is quite general---examples of classes of graphs with boun\-ded expansion include proper minor-closed classes of graphs,
classes of graphs with bounded maximum degree, classes of graphs excluding
a subdivision of a fixed graph, classes of graphs that can be embedded
on a fixed surface with bounded number of crossings per edge and
many others, see~\cite{osmenwood}.  On the other hand, bounded expansion implies a wide range of
interesting structural and algorithmic properties, generalizing many results from proper minor-closed classes of graphs.
For a more in-depth introduction to the topic, the reader is referred to the book of Ne\v{s}et\v{r}il and Ossona de Mendez~\cite{nesbook}.

One of the useful properties of graph classes with bounded expansion is the existence of small balanced separators.
A \emph{separator} of a graph $G$ is a pair $(A,B)$ of subsets of $V(G)$ such that $A\cup B=V(G)$ and no edge joins a vertex of $A\setminus B$
with a vertex of $B\setminus A$.
The \emph{order} of the separator is $|A\cap B|$.
A separator $(A,B)$ is \emph{balanced} if $|A\setminus B|\le 2|V(G)|/3$ and $|B\setminus A|\le 2|V(G)|/3$.
Note that $(V(G),V(G))$ is a balanced separator.
For $c\ge 1$ and $0\le\beta<1$, we say that a graph $G$ has \emph{\bal{c}{\beta}-separators} if every subgraph $H$ of $G$
has a balanced separator of order at most $c|V(H)|^\beta$.
For a graph class $\CC$, let $s_\CC(n)$ denote the smallest nonnegative integer such that every graph in $\CC$ with at most $n$ vertices
has a balanced separator of order at most $s_\CC(n)$.  We say that $\CC$
has \emph{strongly sublinear separators} if there
exist $c\ge 1$ and $0<\delta\le 1$ such that $s_\CC(n)\le cn^{1-\delta}$ for every $n\ge 0$.  Note that
if $\CC$ is subgraph-closed, this implies that every graph in $\CC$ has \bal{c}{1-\delta}-separators.

Lipton and Tarjan~\cite{lt79} proved that the class $\PP$ of planar graphs satisfies $s_{\PP}(n)=O(\sqrt{n})$,
and demonstrated the importance of this fact in the design of algorithms~\cite{lt80}.
This result was later generalized to graphs embedded on other surfaces~\cite{gilbert} and all proper minor-closed
graph classes~\cite{alon1990separator,kreedsep}.
The following result by Plotkin, Rao, and Smith connects the expansion and separators.

\begin{theorem}\label{thm-plotkin}
Given a graph $G$ with $m$ edges and $n$ vertices, and integers $l$ and $h$, there is an $O(mn/l)$-time algorithm
that finds either an $(l\log_2 n)$-minor of $K_h$ in $G$, or a balanced separator of order at most
$O(n/l + lh^2\log n)$.
\end{theorem}

Using this result, Ne\v{s}et\v{r}il and Ossona de Mendez~\cite{grad2} observed that graphs with expansion bounded by subexponential function
have separators of sublinear order.  
The bound on expansion is tight in the sense that 3-regular expanders (which have exponential expansion) do not have sublinear separators.
In fact, polynomial expansion implies strongly sublinear separators, which qualitatively generalizes the results of~\cite{lt79,gilbert,kreedsep}.

\begin{corollary}\label{cor-subsep}
For any $d\ge 0$ and $k\ge 1$, there exists $c\ge 1$ and $\delta=\frac{1}{4d+3}$ such that if the expansion of a graph $G$ is bounded by $f(r)=k(r+1)^d$,
then $G$ has \bal{c}{1-\delta}-separators.  Furthermore, there exists an algorithm that returns a balanced separator of $G$ of order at
most $c|V(G)|^{1-\delta}$ in time $O(|V(G)|^\delta|E(G)|)$.
\end{corollary}
\begin{proof}
For any integer $n\ge 1$, let $l(n)=\lceil n^{\delta}\rceil$ and $h(n)=\lceil n^{1/4-\delta/2}\rceil$.
Since $f(l(n)\log_2 n)=O(n^{d\delta}\log^d n)$ and $1/4-\delta/2>d\delta$, there exists $n_0\ge 1$
such that $f(l(n)\log_2 n)<\frac{h(n)-1}{2}$ for every $n\ge n_0$.

Consider any subgraph $G'$ of $G$, and let $n=|V(G')|$.  Since the expansion of $G$ is bounded by $f(r)$, the expansion of
$G'$ is bounded by $f(r)$ as well.  We aim to show that $G'$ has a balanced separator of order $O(n^{1-\delta})$.
Without loss of generality, we can assume that $n\ge n_0$.
We apply Theorem~\ref{thm-plotkin} to $G'$ with $l=l(n)$ and $h=h(n)$.
Every $(l\log_2 n)$-minor of $G'$ has edge density at most $f(l\log_2 n)<\frac{h-1}{2}$, and thus $G'$ does not contain $K_h$
as an $(l\log_2 n)$-minor.  Consequently, the algorithm of Theorem~\ref{thm-plotkin} produces a balanced separator of order
$O(n/l+lh^2\log n)=O(n^{1-\delta}+n^{1/2}\log n)=O(n^{1-\delta})$.
\end{proof}

Our main result is the converse: in subgraph-closed classes, strongly sublinear separators imply polynomial expansion.

\begin{theorem}\label{thm-expansion}
For any $c\ge 1$ and $0< \delta\le 1$, there exists a function $f(r)=O\bigl(r^{5/\delta^2}\bigr)$ such that
if a graph $G$ has \bal{c}{1-\delta}-separators, then its expansion is bounded by $f$.
\end{theorem}

Let us remark that a weaker variant of Theorem~\ref{thm-expansion} was conjectured by Dvo\v{r}\'ak~\cite{twd},
who hypothesized that strongly sublinear separators imply subexponential expansion, and proved this weaker claim
under the additional assumption that $G$ has bounded maximum degree.
Together with Corollary~\ref{cor-subsep}, Theorem~\ref{thm-expansion} shows the equivalence between strongly sublinear separators and polynomial expansion.

\begin{corollary}
Let $\CC$ be a subgraph-closed class.  Then $\CC$ has strongly sublinear separators if and only if $\CC$ has polynomial expansion.
\end{corollary}

Note that to guarantee separators of order $O(n^{1-\delta})$, Corollary~\ref{cor-subsep} only requires the expansion to be bounded by $r^{O(1/\delta)}$,
while given separators of order $O(n^{1-\delta})$, Theorem~\ref{thm-expansion} guarantees the expansion bounded by $r^{O(1/\delta^2)}$.
For a cubic graph $G$, let $G_\delta$ denote the graph obtained from $G$ by subdividing each edge exactly $\lfloor |V(G)|^{\delta/(1-\delta)}\rfloor$ times,
and let $$\CC_\delta=\{H:\text{$H\subseteq G_\delta$ for a cubic graph $G$}\}.$$  Then balanced separators in $\CC_\delta$ have order $\Omega(n^{1-\delta})$
and the expansion of $\CC_\delta$ is $r^{O(1/\delta)}$.  Hence, the relationship between the exponents in Corollary~\ref{cor-subsep}
is tight up to constant multiplicative factors.  On the other hand, we believe Theorem~\ref{thm-expansion} can be improved.

\begin{conjecture}\label{conj-lindep}
There exists $k>0$ such that for any $c\ge 1$ and $0< \delta\le 1$, if a graph $G$ has \bal{c}{1-\delta}-separators,
then its expansion is bounded by $f(r)=O\bigl(r^{k/\delta}\bigr)$.
\end{conjecture}

The rest of the paper is devoted to the proof of Theorem~\ref{thm-expansion}.  In Section~\ref{sec-expand}, we recall some results
relating separators with tree-width and expanders.
In Section~\ref{sec-dens}, we give partial results towards bounding the density of minors of graphs with strongly sublinear separators. 
In Section~\ref{sec-sep}, we show that a bounded-depth minor of a graph with strongly sublinear separators still has strongly sublinear separators (for a somewhat
worse bound on their order).  Finally, in Section~\ref{sec-poly}, we combine these results to give a proof of
Theorem~\ref{thm-expansion}.

\section{Separators, tree-width and expanders}\label{sec-expand}

For $\alpha>0$, a graph $G$ is an \emph{$\alpha$-expander} if for every $A\subseteq V(G)$ of size at most $|V(G)|/2$, there
exist at least $\alpha|A|$ vertices of $V(G)\setminus A$ adjacent to a vertex of $A$.  Random graphs are asymptotically almost surely expanders.

\begin{lemma}[Kolesnik and Wormald~\cite{kolesnik2014lower}]\label{lemma-exexp}
There exists an integer $n_0$ such that for every even $n\ge n_0$, there exists a $3$-regular $\frac{1}{7}$-expander on $n$ vertices.
\end{lemma}

Let us recall a well-known fact on the relationship between tree-width and separators, see e.g.~\cite{rs2}.

\begin{lemma}\label{lemma-wtsep}
Any graph $G$ has a balanced separator of order at most $\tw(G)+1$.
\end{lemma}

\begin{corollary}\label{cor-twexp}
If $H$ is an $\alpha$-expander for some $\alpha>0$, then $\tw(H)\ge \frac{\alpha}{3(1+\alpha)}|V(H)|-1$.
\end{corollary}
\begin{proof}
Let $(A,B)$ be a balanced separator of $H$ of order at most $\tw(H)+1$.  Let $S=A\cap B$, let $A'=A\setminus B$ and let $B'=B\setminus A$.
Without loss of generality, $|A'|\le |B'|$, and thus $|A'|\le |V(H)|/2$.  Since $H$ is an $\alpha$-expander, we have $|S|\ge\alpha |A'|$, and thus $|A'|\le \frac{1}{\alpha}|S|$.
On the other hand, since the separator $(A,B)$ is balanced, we have $|B'|\le \frac{2}{3}|V(H)|$, and thus $|A'|+|S|\ge \frac{1}{3}|V(H)|$.
Therefore,
\begin{align*}
\left(\frac{1}{\alpha}+1\right)|S|&\ge \frac{1}{3}|V(H)|\\
|S|&\ge \frac{\alpha}{3(1+\alpha)}|V(H)|.
\end{align*}
The claim follows, since $|S|\le \tw(H)+1$.
\end{proof}

For later use, let us remark that an approximate converse to Lemma~\ref{lemma-wtsep} holds, as was proved by
Dvo\v{r}\'ak and Norin~\cite{dnorin}.

\begin{theorem}\label{thm-septw}
If every subgraph of $G$ has a balanced separator of order at most $k$, then $G$ has tree-width at most $105k$.
\end{theorem}

\begin{corollary}\label{cor-septotw}
For $c\ge 1$ and $0\le\beta<1$, if a graph $G$ has \bal{c}{\beta}-separators, then every subgraph $H$ of $G$ has tree-width
at most $105c|V(H)|^\beta$.
\end{corollary}

The aim of this section is to argue that in a dense graph, we can always find a large expander of maximum degree $3$ as a subgraph.  
To prove this, we first find a bounded-depth clique minor, using the following result of Dvo\v{r}\'ak~\cite{twd}.

\begin{theorem}\label{cor-iter2}
Suppose that $0<\varepsilon\le 1$ and let $m=\left\lceil\frac{1}{2\varepsilon^2}\right\rceil$.
If a graph on $n$ vertices has at least $2\cdot 32^mt^4n^{1+\varepsilon}$ edges,
then it contains $K_t$ as a $4^m$-minor.
\end{theorem}

Next, we take a $3$-regular expander subgraph of this clique which exists by Lemma~\ref{lemma-exexp}, and we
observe that it corresponds to a subgraph of the original graph of maximum degree $3$, in which each path of vertices of degree two
has bounded length.  Such a subgraph is still a decent expander.  However, we will only need the corresponding lower bound for the tree-width of such
a graph. 

\begin{lemma}\label{lemma-hdtw}
Suppose that $0<\varepsilon\le 1$ and let $m=\left\lceil\frac{1}{2\varepsilon^2}\right\rceil$.
Let $n_0$ satisfy Lemma~\ref{lemma-exexp} and let $t\ge \max(n_0,600)$ be an even integer.
If a graph on $n$ vertices has at least $2\cdot 32^mt^4n^{1+\varepsilon}$ edges,
then it contains a subgraph $H$ of maximum degree $3$ with $|V(H)|\le 4^{m+1}t$ and with
$\tw(H)\ge \frac{t}{25}$.
\end{lemma}
\begin{proof}
By Theorem~\ref{cor-iter2}, $G$ contains $K_t$ as a $4^m$-minor, and by Lemma~\ref{lemma-exexp}, $G$ contains a $3$-regular $\frac{1}{7}$-expander $H_0$ with $t$ vertices
as a $4^m$-minor.  Hence, $G$ contains a subgraph $H$ of maximum degree three such that $|V(H)|\le (3\cdot 4^m+1)|V(H_0)|\le 4^{m+1}t$ and $H_0$ is a minor of $H$.
By Corollary~\ref{cor-twexp}, $\tw(H)\ge\tw(H_0)\ge \frac{1}{24}|V(H_0)|-1=\frac{t}{24}-1\ge \frac{t}{25}$.
\end{proof}

\section{The densities of graphs with strongly sublinear separators and their minors}\label{sec-dens}

In this section, we give two bounds on edge densities of graphs.  Firstly, we show that graphs with strongly sublinear
separators have bounded edge density; in other words, they satisfy the condition from the definition of bounded expansion
for $\nabla_0$.

\begin{lemma}\label{lemma-dens}
For any $c\ge 1$ and $0<\delta\le 1$, let $a_\delta(c)\ge 1$ be the unique real number satisfying
$a_\delta(c)^\delta=4c\log^2 (ea_\delta(c))$.  If $G$ has \bal{c}{1-\delta}-separators, then $G$ has at most $a_\delta(c)|V(G)|$ edges.
\end{lemma}
\begin{proof}
Let $h(n)=\frac{n}{2\log en}$.  By induction on the number of vertices of $G$, we prove a stronger claim:  If $G$ has \bal{c}{1-\delta}-separators, then $G$ has at most $a_\delta(c)(|V(G)|-h(|V(G)|))$ edges.
Note that $a_\delta(c)(|V(G)|-h(|V(G)|))\ge a_\delta(c)|V(G)|/2$, and thus the claim trivially holds if $|V(G)|\le a_\delta(c)$.

Suppose that $|V(G)|> a_\delta(c)$.  Let $(A,B)$ be a balanced separator of $G$ of order at most $c|V(G)|^{1-\delta}$ and let $G_1=G[A]$ and $G_2=G[B]$.
Let $n=|V(G)|$, $n_0=|V(G_1\cap G_2)|$, $n_1=|V(G_1)\setminus V(G_2)|$ and $n_2=|V(G_2)\setminus V(G_1)|$.
Since $n>a_\delta(c)$, we have $n/3>cn^{1-\delta}\ge n_0$.
For $i\in\{1,2\}$, we have $|V(G_i)|=n_0+n_i<n/3+2n/3<n$, and thus $|E(G_i)|\le a_\delta(c)(n_0+n_i-h(n_0+n_i))$ by the induction hypothesis.
It follows that
\begin{align*}
|E(G)|&\le|E(G_1)|+|E(G_2)|\\
&\le a_\delta(c)(n+n_0-h(n_0+n_1)-h(n_0+n_2))\\
&=a_\delta(c)(n-h(n)+[n_0+h(n)-h(n_0+n_1)-h(n_0+n_2)]).
\end{align*}
Therefore, we need to prove that $n_0\le h(n_0+n_1)+h(n_0+n_2)-h(n)$.
Recall that $n_0\le cn^{1-\delta}$, $n_1,n_2\le 2n/3$, and $n=n_0+n_1+n_2> a_\delta(c)$.
Without loss of generality, $n_1\le n_2$. Since $h$ is increasing and concave for $n\ge 3$, and since $n_0+n_1=n-n_2\ge n/3\ge 3$, we have
$h(n_0+n_1)+h(n_0+n_2)\ge h(n/3)+h(2n/3+n_0)\ge h(n/3)+h(2n/3)$.
We conclude that
\begin{align*}
h(n_0+n_1)&{}+h(n_0+n_2)-h(n)\\
&\ge h(n/3)+h(2n/3)-h(n)\\
&=\frac{n}{6}\left(\frac{1}{\log(en)-\log(3)}+\frac{2}{\log(en)-\log(3/2)}-\frac{3}{\log(en)}\right)\\
&\ge\frac{n}{6\log^2 en}((\log(en)+\log(3))+2(\log(en)+\log(3/2))-3\log(en))\\
&\ge\frac{n}{4\log^2 en}=\frac{n^\delta}{4\log^2 en}n^{1-\delta}\ge\frac{a_\delta(c)^\delta}{4\log^2 (ea_\delta(c))}n^{1-\delta}\\
&=cn^{1-\delta}\ge n_0,
\end{align*}
as required.
\end{proof}
Let us remark that for a fixed $\delta>0$, we have $a_\delta(c)=O\bigl((c\log^3 c)^{1/\delta}\bigr)$.

Secondly, we aim to show that the density of bounded-depth minors of graphs with strongly sublinear separators
grows only slowly with the number of their vertices---slower than $n^\varepsilon$ for every $\varepsilon>0$.
Of course, we will eventually show that it is actually bounded by a constant, but we will need this auxiliary result to do so.

\begin{lemma}\label{lemma-subpolyden}
For any $c\ge 1$, $0<\delta\le 1$, $0<\varepsilon\le 1$ and $r\ge 1$, let $m=\left\lceil\frac{1}{2\varepsilon^2}\right\rceil$,
let $n_0$ satisfy Lemma~\ref{lemma-exexp}, let
$t$ be the smallest even integer greater than $\max\Bigl(n_0,(42000c4^mr)^{1/\delta}\Bigr)$ and
let $b_{c,\delta,\varepsilon}(r)=2\cdot 32^mt^4$. If $G$ has \bal{c}{1-\delta}-separators, then every $r$-minor $F$ of $G$ has less than
$b_{c,\delta,\varepsilon}(r)|V(F)|^{1+\varepsilon}$ edges.
\end{lemma}
\begin{proof}
Suppose that $F$ has at least $b_{c,\delta,\varepsilon}(r)|V(F)|^{1+\varepsilon}$ edges.
By Lemma~\ref{lemma-hdtw}, $F$ contains a subgraph $H$ of maximum degree $3$ with $|V(H)|\le 4^{m+1}t$ and with
$\tw(H)\ge \frac{t}{25}$.  Hence, $G$ contains a subgraph $H'$ of maximum degree $3$ with $|V(H')|\le (3r+1)|V(H)|\le 4^{m+2}rt$
such that $H$ is a minor of $H'$, and thus $\tw(H')\ge \tw(H)\ge \frac{t}{25}$.  By Corollary~\ref{cor-septotw}, we have $\tw(H')\le 105c|V(H')|^{1-\delta}$.
Therefore,
\begin{align*}
\frac{t}{25}&\le 105c4^{m+2}rt^{1-\delta}\\
t^\delta&\le 42000c4^mr.
\end{align*}
This contradicts the choice of $t$.
\end{proof}

Let us remark that for fixed $c$, $\delta$, and $\varepsilon$, we have $b_{c,\delta,\varepsilon}(r)=O\bigl(r^{4/\delta}\bigr)$.

\section{Sublinear separators in bounded-depth minors}\label{sec-sep}

We now prove that if $G$ has strongly sublinear separators, then any bounded-depth minor of $G$ also has strongly sublinear
separators.  Together with Lemma~\ref{lemma-dens}, this will give the bound on their density.

\begin{lemma}\label{lemma-sepmin}
For any $c\ge 1$, $0<\delta\le 1$ and $r\ge 1$, let $\varepsilon=\min\left(1,\frac{\delta}{6(1-\delta)}\right)$ and
let $p_{c,\delta}(r)=316c(b_{c,\delta,\varepsilon}(r)r)^{1-\delta}$.  If $G$ has \bal{c}{1-\delta}-separators, then
every $r$-minor of $G$ has \bal{p_{c,\delta}(r)}{1-\frac{5}{6}\delta}-separators.
\end{lemma}
\begin{proof}
Let $H$ be an $r$-minor of $G$.  Since every subgraph of $H$ is also an $r$-minor of $G$,
it suffices to prove that $H$ has a balanced separator of order at most
$p_{c,\delta}(r)|V(H)|^{1-\frac{5}{6}\delta}$.

By Lemma~\ref{lemma-subpolyden}, $H$ has at most $b_{c,\delta,\varepsilon}(r)|V(H)|^{1+\varepsilon}$ edges.
Hence, there exist a subgraph $H'$ of $G$ with at most $2b_{c,\delta,\varepsilon}(r)r|V(H)|^{1+\varepsilon}+|V(H)|\le 3b_{c,\delta,\varepsilon}(r)r|V(H)|^{1+\varepsilon}$
vertices such that $H$ is a minor of $H'$.  By Corollary~\ref{cor-septotw}, we have
$$\tw(H')\le 315c\left(b_{c,\delta,\varepsilon}(r)r|V(H)|^{1+\varepsilon}\right)^{1-\delta}\le 315c(b_{c,\delta,\varepsilon}(r)r)^{1-\delta}|V(H)|^{1-\frac{5}{6}\delta}.$$
Since $H$ is a minor of $H'$, we have $\tw(H)\le \tw(H')$.
Therefore, by Lemma~\ref{lemma-wtsep}, $H$ has a balanced separator of order at most
$$\tw(H')+1\le 315c(b_{c,\delta,\varepsilon}(r)r)^{1-\delta}|V(H)|^{1-\frac{5}{6}\delta}+1\le p_{c,\delta}(r)|V(H)|^{1-\frac{5}{6}\delta}$$
as required.
\end{proof}

Let us remark that for fixed $c$ and $\delta$, we have $p_{c,\delta}(r)=O\bigl(r^{(4/\delta+1)(1-\delta)}\bigr)=O\bigl(r^{4/\delta}\bigr)$.

\section{Polynomial expansion}\label{sec-poly}

Finally, we can prove our main result.

\begin{proof}[Proof of Theorem~\ref{thm-expansion}]
For any $r\ge 1$, every $r$-minor of $G$ has \bal{p_{c,\delta}(r)}{1-\frac{5}{6}\delta}-separators by Lemma~\ref{lemma-sepmin},
where $p_{c,\delta}(r)=O(r^{4/\delta})$.  By Lemma~\ref{lemma-dens}, every $r$-minor of $G$
has edge density at most $a_{\frac{5}{6}\delta}\bigl(p_{c,\delta}(r)\bigr)=O\bigl(p_{c,\delta}(r)^{\frac{5}{4\delta}}\bigr)=O\bigl(r^{5/\delta^2}\bigr)$.
Therefore, $\nabla_r(G)\le O\bigl(r^{5/\delta^2}\bigr)$.
\end{proof}

\bibliographystyle{siam}
\bibliography{twd}

\end{document}